\documentclass{amsart}
\usepackage{amssymb}

\input xy
\xyoption{all}

\newcommand{\w}{\omega}
\newcommand{\e}{\varepsilon}
\newcommand{\IN}{\mathbb N}
\newcommand{\A}{\mathcal A}

\newcommand{\I}{\mathcal I}
\newcommand{\IR}{\mathbb R}
\newcommand{\IQ}{\mathbb Q}
\newcommand{\F}{\mathcal F}
\newcommand{\Ra}{\Rightarrow}
\newcommand{\K}{\mathcal K}
\newcommand{\PP}{\mathcal P}
\newcommand{\cs}{\mathrm{cs}}
\newcommand{\C}{\mathcal C}
\newcommand{\cbox}{\boxdot}

\newcommand{\Pyt}{\mathfrak{P}}

\newtheorem{theorem}{Theorem}[section]
\newtheorem{proposition}[theorem]{Proposition}
\newtheorem{claim}[theorem]{Claim}
\newtheorem{subclaim}[theorem]{Subclaim}
\newtheorem{problem}[theorem]{Problem}
\newtheorem{example}[theorem]{Example}
\newtheorem{corollary}[theorem]{Corollary}
\theoremstyle{definition}
\newtheorem{definition}[theorem]{Definition}
\newtheorem{remark}[theorem]{Remark}

\title{$\mathfrak{P}_0$-spaces}
\author{Taras Banakh}
\address{Ivan Franko National University of Lviv (Ukraine) and
Jan Kochanowski University in Kielce (Poland)}
\email{t.o.banakh@gmail.com}
\thanks{The work has been partially financed by NCN means granted by decision DEC-2011/01/B/ST1/01439. }
\begin{document}

\begin{abstract}
A regular topological space $X$ is defined to be a {\em $\Pyt_0$-space} if it has countable Pytkeev network. A network $\mathcal N$ for $X$ is called a {\em Pytkeev network} if for any point $x\in X$, neighborhood $O_x\subset X$ of $x$ and subset $A\subset X$ accumulating at a $x$ there is a set $N\in\mathcal N$ such that $N\subset O_x$ and $N\cap A$ is infinite.
The class of $\Pyt_0$-spaces contains all metrizable separable spaces and is (properly) contained in the Michael's class of $\aleph_0$-spaces.  It is closed under many topological operations: taking subspaces, countable Tychonoff products,  small countable box-products, countable direct limits, hyperspaces of compact subsets. For an $\aleph_0$-space $X$ and a $\Pyt_0$-space $Y$ the function space $C_k(X,Y)$ endowed with the compact-open topology is a $\Pyt_0$-space. For any sequential $\aleph_0$-space $X$  the free abelian topological group $A(X)$ and the free locally convex linear topological space $L(X)$ both are  $\Pyt_0$-spaces. A sequential space is a $\Pyt_0$-space if and only if it is an $\aleph_0$-space. A  topological space is metrizable and separable if and only if it is a $\Pyt_0$-space with countable fan tightness.
\end{abstract}
\maketitle

In this paper we introduce a new class of generalized metric spaces called $\Pyt_0$-spaces. Those are regular spaces possessing a countable Pytkeev network.
The class of $\Pyt_0$-spaces contains all metrizable separable spaces and is (properly) contained in the class of $\aleph_0$-spaces introduced by E.Michael \cite{Mi}.
In Section~\ref{s1} we shall introduce the notion of a Pytkeev network and study the relation of the class of $\Pyt_0$-spaces to some other classes of generalized metric spaces. Section~\ref{s2} contains our principal result (Theorem~\ref{main}) saying that for an $\aleph_0$-space $X$ and a $\Pyt_0$-space $Y$ the function space $C_k(X,Y)$ is a $\Pyt_0$-space. This extends the classical result of E.~Michael saying that the function space $C_k(X,Y)$ between $\aleph_0$-spaces in an $\aleph_0$-space.  In Section~\ref{s3} Theorem~\ref{main} is used to show that the class of $\Pyt_0$-spaces is closed under countable Tychonoff products and small box-products. Also we show that for any $\Pyt_0$-space $X$ its hyperspace $\exp(X)$ of compact subsets is a $\Pyt_0$-space. In Section~\ref{s4} we detect $\Pyt_0$-spaces among topological groups and (para)topological lops. In the final section~\ref{s5} we detect $\Pyt_0$-spaces among rectifiable spaces and recover the topological structure of sequential rectifiable $\Pyt_0$-spaces.


\section{$\Pyt_0$-spaces and their relation to other generalized metric spaces}\label{s1}

We start with recalling various notions related to networks.

\begin{definition} Let $X$ be a topological space and $x\in X$ be a point.
A family $\mathcal N$ of subsets of $X$ is called
\begin{itemize}
\item a {\em network at} $x$ if for any neighborhood $O_x\subset X$ of $x$ there is a set $N\subset \mathcal N$ such that $x\in N\subset O_x$;
\item a {\em network} if $\mathcal N$ is a network at each point of $X$;
\smallskip

\item a {\em $k$-network} if for each open set $U\subset X$ and compact subset $K\subset U$ there is a finite subfamily $\F\subset\mathcal N$ such that $K\subset\bigcup\F\subset U$;
\smallskip

\item a {\em $\cs^*$-network at} $x$ if for each sequence $(x_n)_{n\in\w}$ in $X$ convergent to  $x\in X$ and each neighborhood $O_x\subset X$ of $x$ there is a set $N\in\mathcal N$ such that $N\subset O_x$ and the set $\{n\in\w:x_n\in N\}$ is infinite;
\item a {\em $\cs^*$-network} if $\mathcal N$ is a $\cs^*$-network at each point $x\in X$;
\smallskip

\item a {\em Pytkeev $\pi$-network at} $x$ if $\mathcal N$ is a network at $x$ and each neighborhood $O_x\subset X$ contains an infinite set $N\in\mathcal N$;
\smallskip

\item a ({\em strict}) {\em Pytkeev network at} $x$ if $\mathcal N$ is a network at $x$ and for each neighborhood $O_x\subset X$ of $x$ and subset $A\subset X$ accumulating at $x$ there is a set $N\in\mathcal N$ such that $ N\subset O_x$, $N\cap A$ is infinite (and $x\in N$);
\item a ({\em strict}) {\em Pytkeev network} if  $\mathcal N$ is a (strict) Pytkeev network at each point $x\in X$.
\end{itemize}
\end{definition}

We say that a subset $A$ of a topological space $X$ {\em accumulates} at a point $x\in X$ if each neighborhood $O_x\subset X$ of $x$ contains infinitely many points of the set $A$. Each accumulation point belongs to the closure $\bar A$ of the set $A$ in $X$. In a $T_1$-space $X$ a point $x\in X$ is an accumulation point of a set $A\subset X$ if and only if $x$ belongs to the closure of $A\setminus\{x\}$.

Spaces possessing countable networks of various sorts have special names.

\begin{definition} A topological space
$X$ is defined
\begin{itemize}
\item to be a {\em cosmic} space if $X$ is a regular space with a countable network;
\item to be an {\em $\aleph_0$-space} if $X$ is a regular space with a countable $k$-network;
\item to be of {\em countable $\cs^*$-character} if $X$ has a countable  $\cs^*$-network at each point $x$;
\item to have {\em the Pytkeev property} if each subspace $A\subset X$ has a countable Pytkeev $\pi$-network at each accumulation point $x\in X$ of $A$;
\item to have {\em the strong Pytkeev property} if $X$ has a countable Pytkeev network  at each point $x\in X$;
\item to be a ({\em strict}) {\em $\Pyt_0$-space} if $X$ is a regular space with a countable (strict) Pytkeev network.
\end{itemize}
\end{definition}

These notions relate as follows (see Proposition~\ref{p1.2}):
$$
\xymatrix{
\mbox{strict $\Pyt_0$-space}\ar@{=>}[d] \ar@{<=>}[r]&\mbox{$\Pyt_0$-space}\ar@{=>}[r]\ar@{=>}[d]&\mbox{$\aleph_0$-space}\ar@{=>}[d]\ar@{=>}[r]&\mbox{cosmic}\\
\mbox{Pytkeev property}&\mbox{stong Pytkeev property}\ar@{=>}[l]\ar@{=>}[r]&\mbox{countable $\cs^*$-character}.
}
$$
The equivalence (strict $\Pyt_0$-space $\Leftrightarrow$ $\Pyt_0$-space) follows from the fact that for any Pytkeev network $\mathcal N$ in a space $X$ the family $\mathcal N\vee \mathcal N=\{A\cup B:A,B\in\mathcal N\}$ is a strict Pytkeev network. 

\medskip

\begin{remark}
The notion of a network is well-known in General Topology (see \cite[\S3.1]{Eng}) and cosmic spaces form an important class of generalized metric spaces, see \cite[\S10]{Gru2}. $k$-Networks and $\aleph_0$-spaces were introduced by E.Michael \cite{Mi} and studied in \cite[\S11]{Gru}, \cite{Tan}, \cite{BBK}. Spaces with countable $\cs^*$-networks were studied in \cite{Gao}, \cite{LT} and with countable $\cs^*$-character in \cite{BZd}, \cite{MSak}. The Pytkeev property was introduced in \cite{Pyt} and studied in \cite{BM}, \cite{FlD}, \cite{Koc}, \cite{MT}, \cite{MSak2}, \cite{PP}, \cite{ST}.
The strong Pytkeev property was introduced by Tsaban and Zdomskyy in \cite{TZ} and studied in \cite{GKL}, \cite{BL}, \cite{GK1}, \cite{GK2}. The notions of a (strict) Pytkeev network and a $\Pyt_0$-space seem to be new. These notions are central objects of study in this paper.
\end{remark}

First we establish some relations between $\cs^*$-networks and $k$-networks.
It is clear that each $k$-network is a $\cs^*$-network. The converse is true for compact-countable networks.
 
A family $\mathcal N$ of subsets of a topological space $X$ is called {\em compact-countable} if for each compact subset $K\subset X$ the family $\{N\in\mathcal N:N\cap K\ne\emptyset\}$ is at most countable.

\begin{proposition}\label{p:new} Each compact-countable $\cs^*$-network in a topological space is a $k$-network.\footnote{In the published version of this paper Proposition~\ref{p:new} was proved only for Hausdorff spaces.}
\end{proposition}

\begin{proof} Let $\mathcal N$ be a compact-countable $\cs^*$-network in a topological space $X$. To show that $\mathcal N$ is a $k$-network, fix an open set $U\subset X$ and a compact subset $K\subset U$. Consider the countable subfamily $\mathcal N'=\{N\in\mathcal N:K\cap N\ne\emptyset,\;N\subset U\}$ and fix its enumeration $\mathcal N'=\{N_k:k\in\w\}$. Observe that the family $\{N_k\cap K\}_{k\in\w}$ is a countable network for the compact space $K$, which implies that $K$ is hereditarily Lindel\"of. By
Corollary 2.2 \cite{AW}, the hereditarily Lindel\"of compact space $K$ is sequentially compact.

We claim that for some $k\in\w$ the compact set $K$ is contained in the finite union $\bigcup_{i\le k}N_i$. Assuming the converse, for every $k\in\w$ we could find a point $x_k\in K\setminus\bigcup_{i\le k}N_i$. By the sequential compactness of $K$, there is increasing number sequence $(k_i)_{i\in\w}$ such that the sequence $(x_{k_i})_{i\in\w}$ converges to some point $x_\infty\in K$. Since $\mathcal N$ is a $\cs^*$-network, there is a set $N\in \mathcal N$ such that $x_\infty\in N\subset U$ and $\{i\in\w:x_{k_i}\in N\}$ is infinite. Taking into account that $N\subset U$, we conclude that $N\in\mathcal N'$ and hence $N=N_m$ for some $m\in\w$. The choice of the sequence $(x_k)$ guarantees that $x_k\notin N=N_m$ for all $k\ge m$, so the set $\{i\in\w:x_{k_i}\in N=N_m\}\subset\{i\in\w:k_i<m\}$ is finite and this is a desired contradiction showing that $K\subset \bigcup_{i\le k}N_i\subset U$ for some $k$, and witnessing that the $\cs^*$-network $\mathcal N$ is a $k$-network for $X$.
\end{proof}

This proposition implies the following characterization of $\aleph_0$-spaces (which improves a bit a characterization \cite{Guth} of $\aleph_0$-spaces as regular spaces with countable $\cs$-network).

\begin{corollary} A regular space $X$ is an $\aleph_0$-space if and only if $X$ has countable $\cs^*$-network.
\end{corollary}

\begin{proposition}\label{p1.2} Each countable Pytkeev network is a $k$-network. Consequently, each $\Pyt_0$-space is an $\aleph_0$-space.
\end{proposition}

\begin{proof} Assume that $\mathcal N$ is a countable Pytkeev network in a topological space $X$.
To show that $\mathcal N$ is a $k$-network, take any open set $U\subset X$ and a compact subset $K\subset U$. Consider the countable subfamily  $\mathcal N(U)=\{N\in\mathcal N:N\subset U\}$ and let $\mathcal N(U)=\{N_k:k\in\w\}$ be its enumeration. We claim that $K\subset N_0\cup\dots\cup N_k$ for some $k\in\w$. In the opposite case, for every $k\in\IN$ we could choose a point $x_k\in K\setminus\bigcup_{i=0}^k N_i$ and observe that the infinite subset $A=\{x_k\}_{k\in\w}$ of the compact space $K$ accumulates at some $x_\infty\in K$. Consequently, $A\cap N_k$ is infinite for some set $N_k\in\mathcal N(U)$, which is not possible as $|N_k\cap A|\le k$. So, $\mathcal N$ is a $k$-network in $X$.
\end{proof}

For sequential spaces or $k_2$-spaces Proposition~\ref{p1.2} can be reversed. Let us recall that a topological space is
\begin{itemize}
\item {\em sequential} if for each non-closed subset $A\subset X$ there is a sequence $\{a_n\}_{n\in\w}\subset A$, convergent to some point $x\in X\setminus A$;
\item a {\em $k$-space} if for each non-closed subset $A\subset X$ there is a compact subset $K\subset X$ such that $K\cap A$ is not closed in $K$;
\item a {\em $k_2$-space} if for each non-closed subset $A\subset X$ there is a compact Hausdorff subspace $K\subset X$ such that $K\cap A$ is not closed in $K$.
\end{itemize}
It is clear that a space is a $k$-space if it is sequential or a $k_2$-space. A $k$-space $X$ is sequential if all compact subsets of $X$ are metrizable. In particular, a cosmic space is sequential if and only if it is a $k$-space.

\begin{proposition}\label{p1.3n} Each $k$-network $\mathcal N$ in a $k_2$-space $X$ is a Pytkeev network.
\end{proposition}

\begin{proof} Let $\mathcal N$ be a $k$-network for a $k_2$-space $X$. To show that $\mathcal N$ is a Pytkeev network, fix a subset $A\subset X$ accumulating at some point $x\in X$. Given any neighborhood $O_x\subset X$, we need to find a set $N\in\mathcal N$ such that $N\subset O_x$ and $N\cap A$ is infinite.

Let $\ddot x$ be the intersection of all neighborhoods of the point $x$. It is clear that $\ddot x$ is a compact subset of $U_x$, so $\ddot x\subset\bigcup\F\subset O_x$ for some finite subfamily $\F\subset \mathcal N$. If the intersection $\ddot x\cap A$ is infinite, then for some $F\in\F\subset\mathcal N$ the intersection $A\cap F$ is infinite and we are done.

So, we assume that $\ddot x\cap A$ is finite. Since $x$ is an accumulating point of $A$, it is also accumulating point of $A\setminus \ddot x$. Replacing the set $A$ by $A\setminus \ddot x$ we can assume that
$A\cap \ddot x=\emptyset$.
It follows that the set $B=A\cup (X\setminus O_x)$ accumulates at $x$ and hence it is not closed in $X$.

Since $X$ is a $k_2$-space, there exists a compact Hausdorff subspace $K\subset X$ such that $K\cap B$ is not closed in $K$. Consequently, there exists a point $z\in K\setminus B\subset O_x$ which belongs to the closure of the set $K\cap B$. The compact space $K$, being Hausdorff, is regular. So, we can choose a compact neighborhood $K_z\subset K\cap O_x$ of the point $z$ in $K$. Observe that the set $K_z\setminus B$ is infinite (since $z\notin K_z\setminus B$ is contained in the closure of $K_z\setminus B$ and $K_z$ is Hausdorff). Since $\mathcal N$ is a $k$-network, for the compact set $K_z\subset O_x$ there is a finite subfamily $\F\subset\mathcal N$ such that $K_z\setminus B\subset K_z\subset\bigcup\F\subset O_x$. By Pigeonhole Principle, for some set $F\in\F\subset\mathcal N$ the set $F\setminus B=F\cap A$ is infinite.
\end{proof}

For sequential spaces we can prove a bit more (but in more difficult way).

\begin{proposition}\label{p1.3s} Each $\cs^*$-network $\mathcal N$ in a sequential space $X$ is a Pytkeev network.
\end{proposition}

\begin{proof} To show that  $\mathcal N$ is a Pytkeev network in $X$, take any set $A\subset X$ accumulating at a point $x\in X$. Given an open neighborhood $O_x\subset X$ of $x$, we need to find a set $N\in\mathcal N$ such that $N\subset O_x$ and $N\cap A$ is infinite. Let $\ddot x$ be the intersection of all neighborhoods of $x$ in $X$. If $\ddot x\cap A$ is infinite, then we can choose a sequence $(x_n)_{n\in\w}$ of pairwise distinct points of the set $\ddot x\cap A$ and observe that this sequence converges to $x$. The family $\mathcal N$, being a $\cs^*$-network at $x$, contain a set $N\in\mathcal N$ such that $N\subset O_x$ and the set $\{n\in\w:x_n\in N\}$ is infinite. Then the intersection $N\cap A$ is infinite as well.

So, we assume that the intersection $\ddot x\cap A$ is finite. In this case the set $A\setminus \ddot x$ also accumulates at $x$. Replacing $A$ by $A\setminus \ddot x$, we can assume that $A\cap \ddot x=\emptyset$. Then the set $B=A\cup(X\setminus O_x)$ contains $x$ in its closure too and does not intersect $\ddot x$ (as $\ddot x\subset O_x$). For any subset $C\subset X$ denote by $C^{(1)}$ the sequential closure of $C$, i.e., the set of limit points of sequences $\{c_n\}_{n\in\IN}\subset C$ that converge in $X$. Let $B^{(0)}=B$ and by transfinite induction, for every ordinal $\alpha>1$ put $B^{(\alpha)}=(B^{(<\alpha)})^{(1)}$ where $B^{(<\alpha)}=\bigcup_{\beta<\alpha}B^{(\beta)}$. The sequentiality of the space $X$ guarantees that $x\in B^{(\alpha_x)}$ for some ordinal $\alpha_x>0$.

 If $x\in B^{(1)}$, then we can choose a sequence $\{b_n\}_{n\in\w}\subset B$, convergent to $x$.
Since $B\cap \ddot x=\emptyset$, this sequence cannot contain a constant subsequence and hence it contains a subsequence of pairwise distinct points of $B$. Replacing the sequence $(b_n)_{n\in\w}$ by this subsequence, we can assume that all points $b_n$, $n\in\w$, are pairwise distinct and belong to the neighborhood $O_x$. Since $\mathcal N$ is a  $\cs^*$-network $\mathcal N$, there is a set $N\in\mathcal N$ such that $N\subset O_x$ and $N$ contains infinitely many points of the convergent sequence $\{b_n\}_{n\in\w}\subset B\cap U_x\subset A$ and hence infinitely many points of the set $A$. Now assume that $x\notin B^{(1)}$.

A sequence of points $(b_n)_{n\in\w}$ will be called {\em injective} if $b_n\ne b_m$ for any distinct numbers $n,m\in\w$.

\begin{claim}\label{cl1.8} For every ordinal $\alpha\ge 1$, any open neighborhood $O_z\subset X$ of any point $z\in B^{(\alpha)}\setminus B^{(1)}$ contains an injective sequence $\{b_n\}_{n\in\w}\subset B$ convergent to a point of $O_z$.
\end{claim}

\begin{proof} This claim will be proved by induction on $\alpha\ge 1$.
For $\alpha=1$ this statement trivially holds. Assume that for some ordinal $\alpha>1$ we have proved the claim for all ordinals $<\alpha$. Choose any point $z\in B^{(\alpha)}\setminus B^{(1)}$ and an open neighborhood $O_z\subset X$ of $z$. If $z\in B^{(\beta)}$ for some $\beta<\alpha$, then we can apply the inductive assumption and complete the proof.  So, assume that $z\notin B^{(<\alpha)}$. In this case $z$ is the limit of some convergent sequence $\{z_n\}_{n\in\w}\subset B^{(<\alpha)}\cap O_z$. If $\alpha>2$, then $z\notin B^{(<\alpha)}$ implies the existence of $n\in\w$ such that $z_n\notin B^{(1)}$. Then by the inductive assumption, the neighborhood $O_z$ of $z_n$ contains an injective sequence $\{b_n\}_{n\in\w}\subset B$ converging to some point of $O_z$ and we are done.

It remains to consider the case $\alpha=2$. Assume that the neighborhood $O_z$ contains no injective sequence $\{b_n\}_{n\in\w}\subset B$ convergent to a point of $O_z$. Since $z\in B^{(2)}\setminus B^{(1)}$, in the sequence $\{z_n\}_{n\in\w}\subset B^{(1)}$ there is only finitely many points of the set $B$. Passing to a subsequence, we can assume that $z_n\notin B$ for all $n\in\w$. Since $O_z$ contains no injective convergent sequence of points of $B$, for every $n\in\w$ the point $z_n\in O_z\cap B$ is the limit of some constant sequence in $B$, which implies that the intersection $\ddot{z}_n$ of all neighborhoods of $z_n$ meets the set $B$ and hence contains a point $b_n\in B\cap \ddot{z}_n$. It is easy to see that the convergence of the sequence $(z_n)_{n\in\w}$ to $z$ implies the convergence of $(b_n)_{n\in\w}$ to the same limit $z$. This means that $z\in B^{(1)}$, which contradicts the choice of $z$.
This contradiction completes the proof of Claim~\ref{cl1.8}.
\end{proof}

By Claim~\ref{cl1.8}, for the ordinal $\alpha_x$ the neighborhood $O_x$ of the point $x\in B^{(\alpha_x)}\setminus B^{(1)}$ contains an injective sequence $\{b_n\}_{n\in\w}\subset B$ convergent to a point $b_\infty\in U_x$. By definition, the $\cs^*$-network $\mathcal N$ contains a set $N\subset O_x$  containing infinitely many points of the sequence $\{b_n\}_{n\in\w}\subset B\cap U_x\subset A$, which implies that $A\cap N$ is infinite.
\end{proof}

Propositions~\ref{p1.2}, \ref{p1.3n} and \ref{p1.3s} imply:

\begin{corollary}\label{c1.7} A $k$-space $X$ is a $\Pyt_0$-space if and only if $X$ is an $\aleph_0$-space if and only if $X$ is a regular space with a countable $\cs^*$-network.
\end{corollary}

The following (probably known) example shows that the class of $\Pyt_0$-spaces is properly contained in the class of $\aleph_0$-spaces.

\begin{example} For any free ultrafilter $p$ on $\IN$ the space $X=\IN\cup\{p\}\subset\beta\IN$
\begin{itemize}
\item has a countable $k$-network and hence is an $\aleph_0$-space;
\item fails to have the Pytkeev property at the point $p\in X$ and hence fails to be a $\Pyt_0$-space.
\end{itemize}
\end{example}

\begin{proof} Since each compact subset of the space $X=\IN\cup\{p\}$ is finite, the family of singletons is a countable $k$-network for $X$.

Next, we show that $X$ does not have the Pytkeev property at the point $p$. It suffices to show that any countable family $\mathcal P=\{P_n\}_{n\in\w}$ of infinite subsets of $X$ fails to be a Pytkeev $\pi$-network at $p$.
For every $n\in\w$ by induction we can choose two distinct points $a_n,b_n\in P_n\setminus\{p, a_k,b_k:k<n\}$. Enlarge the sets $\{a_n\}_{n\in\w}$ and $\{b_n\}_{n\in\w}$ to two disjoint sets $A,B$ such that $A\cup B=\w$. One of the sets $A$ or $B$ belongs to the ultrafilter $p$. If $A\in p$, then for the neighborhood $O_p=\{p\}\cup A$ there is no set $P\in\mathcal P$ with $p\in P\subset O_p$. If $B\in p$, then for the neighborhood $O_p=\{p\}\cup B$ there is no set $P\in\mathcal P$ with $p\in P\subset O_p$. In both cases we conclude that $\mathcal P$ is not a Pytkeev $\pi$-network at $p$. Consequently, $X$ does not have the Pytkeev property and fails to be a $\Pyt_0$-space.
\end{proof}

We shall say that a topological space $X$ has {\em countable fan} ({\em open-}){\em tightness} if for each point $x\in X$ and (open) sets $A_n\subset X$, $n\in\w$, with $x\in\bigcap_{n\in\w}\bar A_n$ there are finite sets $F_n\subset A_n$, $n\in\w$, such that for any neighborhood $O_x\subset X$ of $x$ the set $\{n\in\w:F_n\cap O_x\ne\emptyset\}$ is infinite.
The countable fan tightness is a well-known notion in the topological theory of function spaces (see \cite[II.\S2]{Arh}). Its ``open'' modification was introduced by M.~Sakai in \cite{MSak} (as property $(\#)$~).

\begin{theorem}\label{t1.6} A (regular) topological space $X$ is second countable if and only if $X$ has a countable Pytkeev network and countable fan (open-)tightness.
\end{theorem}

\begin{proof} The ``only if'' part is trivial. To prove the ``if'' part, assume that a (regular) topological space $X$ with countable fan tightness has a countable Pytkeev network $\mathcal N$. We lose no generality assuming that $\mathcal N$ is closed under finite unions.
If $X$ is regular, then we can replace each set $N\in\mathcal N$ by its closure and assume that $\mathcal N$ consists of closed subsets of $X$.

 We claim that the countable family $\mathcal B$ consisting of the interiors $N^\circ$ of the sets $N\in\mathcal N$ is a base of the topology of $X$.

Given any open set $U\subset X$ and a point $x\in U$, we need to find a set $N\in\mathcal N$ with $x\in N^\circ\subset U$. Consider the countable subfamily $\mathcal N_U=\{N\in\mathcal N\colon N\subset U\}$ and let $\mathcal N_U=\{N_k\}_{k\in\w}$ be its enumeration. Assuming that $x\notin N^\circ$ for every $N\in\mathcal N_U$ and taking into account that the family $\mathcal N_U$ is closed under finite unions, we conclude that the every $k\in\w$ the (open) set $A_k=X\setminus \bigcup_{i\le k}N_i$ contains the point $x$ in its closure. By the countable fan (open-)tightness of $X$, there are finite sets $F_k\subset A_k\cap U$, $k\in\w$, such that every neighborhood $O_x$ of $x$ meets infinitely many sets $F_k$, $k\in\w$. Since $\bigcup_{k\in\w}N_k=U$, for every $k\in\w$ there is $n\in\w$ such that $F_k\subset \bigcup_{i\le n}N_i$ and hence $F_k\cap F_m=\emptyset$ for all $m\ge n$. This observation implies that $x$ is an accumulation point of the set $A=\bigcup_{k\in\w}F_k$.

Since $\mathcal N$ is a Pytkeev network, there is a set $N\in\mathcal N$ such that $x\in N\subset U$ and $N\cap A$ is infinite. The set $N$ belongs to the family $\mathcal N_U$ and hence coincides with some set $N_m$. The choice of the sequence $(F_k)$ guarantees that $N\cap A=N\cap\bigcup_{k\in\w}F_k\subset \bigcup_{k<m}F_k$ is finite, which is a desired contradiction showing that $x\in N^\circ\subset N\subset U$ for some $N\in\mathcal N$. So, $\mathcal B=\{N^\circ:N\in\mathcal N\}$ is a countable base of the topology of $X$.
\end{proof}

Taking into account that each regular second countable space is metrizable and separable, we see that Theorem~\ref{t1.6} implies the following characterization of separable metrizable spaces.

\begin{corollary}\label{c:fan-open} A topological space $X$ is separable and metrizable if and only if $X$ is a $\Pyt_0$-space with countable fan open-tightness.
\end{corollary}


\section{Function spaces between $\Pyt_0$-spaces}\label{s2}

One of the most important results related to $\aleph_0$-spaces is Michael's Theorem \cite{Mi} (see also \cite[11.5]{Gru}) saying that for any $\aleph_0$-spaces $X,Y$ the space $C_k(X,Y)$ of continuous functions endowed with the compact-open topology is an $\aleph_0$-space. If $Y=\IR$, then we shall write $C_k(X)$ instead of $C_k(X,\IR)$.

It turns out that this Michael's theorem admits an extension to $\Pyt_0$-spaces: for any $\aleph_0$-space $X$ and a $\Pyt_0$-space $Y$ the function space $C_k(X,Y)$ is a $\Pyt_0$-space. In fact, we prove this result not only for the compact-open topology on $C(X,Y)$ but for a wide class of so-called $\I$-open topologies determined by ideals $\I$ of compact sets.

Let $X$ be a topological space. A family $\I$ of compact subsets of $X$ is called {\em an ideal of compact sets} if $\bigcup\I=X$ and for any sets $A,B\in\I$ and compact subset $K\subset X$ we get $A\cup B\in\I$ and $A\cap K\in\I$.

For an ideal $\I$ of compact subsets of a topological space $X$ and a topological space $Y$ by $C_\I(X,Y)$ we shall denote the space $C(X,Y)$ of all continuous functions from $X$ to $Y$, endowed with the {\em $\I$-open topology} generated by the subbase consisting of the sets
$$[K;U]=\{f\in C_\I(X,Y):f(K)\subset U\}$$where $K\in\I$ and $U$ is an open subset of $Y$.

If $\I$ is the ideal of all compact (resp. finite) subsets of $X$, then the $\I$-open topology coincides with the compact-open topology (resp. the topology of pointwise convergence) on $C(X,Y)$. In this case the function space $C_\I(X,Y)$ will be denoted by $C_k(X,Y)$ (resp. $C_p(X,Y)$).

We shall be interested in detecting $\Pyt_0$-spaces among function spaces $C_\I(X,Y)$. For this we should impose some restrictions on the ideal $\I$.

\begin{definition} An ideal $\I$ of compact subsets of $X$ is defined to be {\em discretely-complete} if for any compact subsets $A,B\subset X$ such that $A\setminus B$ is a countable discrete subspace of $X$ the inclusion $B\in\I$ implies $A\in\I$.
\end{definition}

It is clear that the ideal of all compact subsets of $X$ is discretely-complete. More generally, for any infinite cardinal $\kappa$ the ideal $\I$ of compact subsets of cardinality $\le\kappa$ in  $X$ is discretely-complete. On the other hand, the ideal of finite subsets of $X$ is discretely-complete if and only if $X$ contains no infinite compact subset with finite set of non-isolated points.

\begin{theorem}\label{main} Let $X$ be a Hausdorff space with countable $k$-network and $\I$ be a discretely-complete ideal of compact subsets of $X$. For any  topological space $Y$ with a countable Pytkeev network, the function space $C_\I(X,Y)$ has a countable Pytkeev network too.
\end{theorem}

\begin{proof}
Let $\mathcal K$ be a countable $k$-network on the space $X$ and $\mathcal P$ be a countable Pytkeev network on the space $Y$. We lose no generality assuming that the networks $\mathcal K$ and $\mathcal P$ are closed under finite unions and finite intersections.

For two subsets $K\subset X$ and $P\subset Y$ let
$$[K;P]=\{f\in C_\I(X,Y):f(K)\subset P\}\subset C_\I(X,Y).$$
We claim that the countable family
$$[\kern-1.5pt[\mathcal K;\mathcal P]\kern-1.5pt]=\{[K_1;P_1]\cap\dots\cap [K_n;P_n]:n\in\IN,\;K_1,\dots,K_n\in\K,\;P_1,\dots,P_n\in\mathcal P\}$$
is a strict Pytkeev network for the function space $C_\I(X,Y)$.

Given a subset $\A\subset C_\I(X,Y)$, a function $f\in\bar\A$ and a neighborhood $O_f\subset C_\I(X,Y)$ of $f$ we need to find a set $\F\in[\kern-1.5pt[\K,\PP]\kern-1.5pt]$ such that $f\in \F\subset O_f$, and moreover $\A\cap \F$ is infinite if $f$ is an accumulation point of the set $\A$.

We lose no generality assuming that $\A\subset O_f$ and the neighborhood $O_f$ is of basic form
$$O_f=[C_1;U_1]\cap\dots\cap [C_n;U_n]$$
for some compact sets $C_1,\dots,C_n\subset X$ and some open sets $U_1,\dots,U_n\subset Y$.

For every $i\le n$ consider the countable subfamily $\K_i=\{K\in\K:C_i\subset K\subset f^{-1}(U_i)\}$ and let $\K_i=\{K'_{i,j}\}_{j\in\w}$ be its enumeration. For every $j\in\w$ let
$K_{i,j}=\bigcap_{k\le j}K'_{i,k}$. It follows that the decreasing sequence of sets  $(K_{i,j})_{j=1}^\infty$ converges to $C_i$ in the sense that each neighborhood $O(C_i)\subset X$ contains all but finitely many sets $K_{i,j}$ for $j\in\IN$.

For every $i\le n$ consider the countable subfamily $\mathcal P_i=\{P\in\mathcal P:f(C_i)\subset P\subset U_i\}$ in the Pytkeev network $\mathcal P$ and let $\mathcal P_i=\{P'_{i,j}\}_{j\in\w}$ be its enumeration. For every $j\in\w$ let $P_{i,j}=\bigcup_{k\le j}P_{i,k}'$. It follows that $(P_{i,j})_{j\in\w}$ is an increasing sequence of sets with $f(C_i)\subset \bigcap_{j\in\w}P_{i,j}\subset \bigcup_{j\in\w}P_{i,j}=U_i$.

Then the sets
$$\F_j=\bigcap_{i=1}^n[K_{i,j};P_{i,j}]\in[\kern-1.5pt[\mathcal K;\mathcal P]\kern-1.5pt]$$form an increasing sequence of sets in the function space $C_{\I}(X,Y)$.

\begin{claim}\label{cl1} $\bigcup_{j\in\w}\F_j=O_f=\bigcap_{i=1}^n[C_i;U_i]$.
\end{claim}

\begin{proof} Assume that some function $g\in O_f=\bigcap_{i=1}^n[C_i,U_i]$ does not belong to $\bigcup_{j\in\w}\F_j$. Then for every $j\in\w$ we can find an index $i_j\in\{1,\dots,n\}$ such that $g\notin[K_{i_j,j},P_{i_j,j}]$, which means that $g(x_j)\notin P_{i_j,j}$ for some point $x_j\in K_{i_j,j}$. By the Pigeonhole Principle, for some $i\in\{1,\dots,n\}$ the set $J_i=\{j\in\w:i_j=i\}$ is infinite. Since the decreasing sequence $(K_{i,j})_{j\in J_i}$ converges to the compact set $C_i$, the set $C_i\cup\{x_j\}_{j\in J_i}$ is compact. Moreover the compact set $C_i$ has countable network and hence is metrizable.

Then we can find an infinite subset $J'\subset J_i$ such that the sequence $\{x_j\}_{j\in J'}$ converges to some point $x_\infty\in C_i$. By the continuity of the function $g$, the sequence $\{g(x_j)\}_{j\in J'}$ converges to the point $g(x_\infty)\in g(C_i)\subset U_i$. Since each countable Pytkeev network is a $k$-network (see Proposition~\ref{p1.2}), there is a finite subfamily $\mathcal P'\subset\mathcal P$ such that $f(C_i)\cup\{g(x_j)\}_{j\in J'}\subset\bigcup\mathcal P'\subset U_i$. The set $P'=\bigcup\mathcal P'\in\mathcal P$ belongs to the family $\mathcal P_i$ and hence $P'=P'_{i,j'}$ for some number $j'$. Then for every $j\in J'$ with $j\ge j'$ we get $g(x_{j})\in P'\subset P_{i,j}$ which contradicts the choice of the point $x_j$. This contradiction completes the proof of Claim~\ref{cl1}.
\end{proof}

The following claim completes the proof of Theorem~\ref{main}.

\begin{claim}\label{cl2} If $f$ is an accumulation point of $\A$, then  for some $j\in\w$ the intersection $\F_j\cap \A$ is infinite.
\end{claim}

\begin{proof} Assume conversely that for every $j\in\w$ the intersection $\A_j=\F_j\cap\A$ is finite.
Then $\A=\A\cap O_f=\bigcup_{j\in\w}\A\cap\F_j=\bigcup_{j\in\w}\A_j$ is the countable union of an increasing sequence $(\A_j)_{j\in\w}$ of finite subsets of $C_k(X,Y)$.

To each function $\alpha\in\A\setminus\A_0$ assigns a (unique) number $j_\alpha\in\w$ such that $\alpha\in\A_{j_\alpha+1}\setminus\A_{j_\alpha}=\A_{j_\alpha+1}\setminus\F_{j_\alpha}$. Since $\alpha\notin\F_{j_\alpha}=\bigcap_{i=1}^n[K_{i,j_\alpha};P_{i,j_\alpha}]$, there are an index  $i_\alpha\in\{1,\dots,n\}$ such that $\alpha\notin[K_{i_\alpha,j_\alpha},P_{i_\alpha,j_\alpha}]$ and a point $x_\alpha\in K_{i_\alpha,j_\alpha}$ with $\alpha(x_\alpha)\notin P_{i_\alpha;j_\alpha}$.

For every $i\in\{1,\dots,n\}$ consider the subfamily
$$\A(i)=\{\alpha\in \A\setminus\A_0:i_\alpha=i\}$$ and observe that $\A\setminus\A_0=\bigcup_{i=1}^n\A(i)$.
Consider the subset $B_i=\{\alpha(x_\alpha):\alpha\in \A(i)\}\subset Y$. We claim that no point $y\in f(C_i)$ is an accumulation point of the set $B_i$ in $Y$. Assuming conversely that some point $y\in f(C_i)$ is an accumulation point of the set $B_i$, we could find a set $P\in\mathcal P$ such that $P\subset U_i$ and $P\cap B_i$ is infinite. By Proposition~\ref{p1.2}, the countable Pytkeev network $\mathcal P$ is a $k$-network in $Y$, so there is a set $P'\in\mathcal P$ such that $f(C_i)\subset P'\subset U_i$ and then the set $P_y=P\cup P'$ belongs to the family $\mathcal P_i$ and hence coincides with some set $P_{i,j}'\subset P_{i,j}$. The choice of the point $x_\alpha$ guarantees that $\alpha(x_\alpha)\notin P_{i,j}$ for all $\alpha\in \A(i)\setminus\A_{j}$, which implies that the intersection $B_i\cap P\subset B_i\cap P_y\subset\{\alpha(x_\alpha):\alpha\in\A(i)\cap\A_j\}$ is finite. But this contradicts the choice of the set $P$.

This contradiction shows that each point $y\in f(C_i)$ has an open neighborhood $O_y\subset U_i$ with finite intersection $O_y\cap B_i$. By the compactness of $f(C_i)$ the open cover $\{O_y:y\in f(C_i)\}$ of $f(C_i)$ has a finite subcover $\{O_y:y\in F_i\}$ and then $V_i=\bigcup_{y\in F_i}O_y\subset U_i$ is an open neighborhood of $f(C_i)$ having finite intersection with the set $B_i$.

Since the decreasing sequence $(K_{i,j})_{j\in\w}$ converges to $C_i$, there is a number $j_i\in\w$ such that $K_{i,j_i}\subset f^{-1}(V_i)$. We can take $j_i$ so large that $V_i\cap\{\alpha(x_\alpha):\alpha\in\A(i)\setminus\A_{j_i}\}=\emptyset$. It follows that
$$C_i'=C_i\cup\{x_\alpha:\alpha\in \A(i)\setminus\A_{j_i}\}\subset K_{i,j_i}$$is a compact subset of $f^{-1}(V_i)$.
Since the complement $C_i'\setminus C_i\subset\{x_\alpha:\alpha\in\A(i)\setminus\A_{j_i}\}$ is countable and consists of isolated points, the discrete-completeness of the ideal $\I$ guarantees that $C'_i\in\I$.

Then
$$V_f=\bigcap_{i=1}[C_i';V_i]$$ is an open neighborhood of $f$ in $C_\I(X,Y)$. Observe that for every $i\in\w$ and $\alpha\in\A(i)\setminus\A_{j_i}$, we get $x_\alpha\in C_i'$ and $\alpha(x_\alpha)\in Y\setminus V_i$, so $\alpha\notin V_f$. Since the set $\A'=\bigcup_{i=1}^n\A(i)\setminus\A_{j_i}$ has finite complement $\A\setminus\A'\subset\bigcup_{i=1}^n\A_{j_i}$, the intersection $\A\cap V_f$ is finite, which is not possible as $f$ is an accumulation point of $\A$. This contradiction completes the proof of Claim~\ref{cl2}.
\end{proof}

Claims~\ref{cl1}, \ref{cl2} witness that the countable family  $[\kern-1.5pt[\mathcal K;\mathcal P]\kern-1.5pt]$ is a strict Pytkeev network for the function space $C_\I(X,Y)$.
\end{proof}

\begin{corollary}\label{tipa} For any separable metrizable spaces $X,Y$ the function space $C_k(X,Y)$ is a $\Pyt_0$-space.
\end{corollary}

\begin{corollary}\label{tipa2} For any $\aleph_0$-space $X$ the function space $C_k(X)=C_k(X,\mathbb R)$ is a $\Pyt_0$-space.
\end{corollary}

\begin{remark} Corollary~\ref{tipa2} implies that for each $\aleph_0$-space $X$ the function space $C_k(X)$ has the strong Pytkeev property, which answers a problem posed in \cite{GKL}. For every Polish space $X$ the strong Pytkeev property of the function spaces $C_k(X)$ was proved (by a totally different method) in \cite{TZ}.
\end{remark}

\begin{remark} In contrast to function spaces $C_k(X)$, the function spaces $C_p(X)$ endowed with the topology of pointwise convergence rarely are $\Pyt_0$-spaces. By \cite{MSak}, for a Tychonoff space $X$ the following conditions are equivalent:
\begin{enumerate}
\item $X$ is countable;
\item $C_p(X)$ has countable $\cs^*$-character;
\item $C_p(X)$ has the strong Pytkeev property;
\item $C_p(X)$ is a $\Pyt_0$-space;
\item $C_p(X)$ is an $\aleph_0$-space;
\item $C_p(X)$ is metrizable.
\end{enumerate}
This shows that the discrete-completeness of the ideal $\I$ is essential in Theorem~\ref{main}.
\end{remark}

\begin{remark}
The paper \cite{MSak2} contains a characterization of Tychonoff spaces $X$ whose function spaces $C_p(X)$ have the Pytkeev property. This characterization implies that if for a subspace $X\subset\IR$ the function space $C_p(X)$ has the Pytkeev property, then $X$ has universal measure zero and is perfectly meager.
\end{remark}

Theorem~\ref{main} combined with Theorem~\ref{t1.6} allows us to characterize $\aleph_0$-spaces $X$ whose function spaces $C_k(X)$ have countable fan (open-)tightness.
 Let us recall that a topological space $X$ is {\em hemicompact} if there is a countable family $\mathcal K$ of compact subsets of $X$ such that each compact subset $C\subset X$ is contained in some set $K\in\mathcal K$. For first-countable Lindel\"of spaces the following characterization is proved in Corollary 3.6 of \cite{MSak}.

\begin{corollary} For an $\aleph_0$-space $X$ the following conditions are equivalent:
\begin{enumerate}
\item $C_k(X)$ has countable fan tightness;
\item $C_k(X)$ has countable fan open-tightness;
\item $C_k(X)$ is metrizable;
\item $C_k(X)$ is first countable;
\item $X$ is hemicompact.
\end{enumerate}
\end{corollary}

\begin{proof} By Corollary~\ref{tipa2}, the function space $C_k(X)$ is a $\Pyt_0$-space. The implication $(1)\Ra(2)$ is trivial.
If $C_k(X)$ has countable fan open-tightness, then by Corollary~\ref{c:fan-open} it is metrizable. This proves the implication $(2)\Ra(3)$. The implication $(3)\Ra(1)$ is trivial. The equivalences $(3)\Leftrightarrow(4)\Leftrightarrow(5)$ are proved in the classical paper of Arens \cite{Arens}.
\end{proof}

\begin{remark} By Theorem II.2.2 \cite{Arh}, for a Tychonoff space $X$ the function space $C_p(X)\subset \IR^X$ has countable fan tightness if and only if all powers $X^n$, $n\in\IN$, have the Hurewicz property.
\end{remark}

\begin{remark}\label{MSak-Cp} By  Theorem 3.8 \cite{MSak}, for any topological space $X$ the function space $C_p(X)$ has countable fan open-tightness. More precisely, for any sequence $(U_n)_{n\in\w}$ of open subsets of $C_p(X)$ and any point $f\in\bigcap_{n\in\w}\bar U_n$ there is a sequence of functions $f_n\in U_n$, $n\in\w$ such that $f$ is a cluster point of the set $\{f_n\}_{n\in\w}$.
\end{remark}

\section{Preservation of the class of $\Pyt_0$-spaces by some topological operations}\label{s3}

In this section we shall show that the class of $\Pyt_0$-spaces is closed under many operations over topological spaces and hence is a nice new class of generalized metric spaces.

The following trivial proposition implies that the class of $\Pyt_0$-spaces is closed under taking subspaces.

\begin{proposition}\label{p3.1} If $\mathcal N$ is a Pytkeev network for a topological space $X$, then for every subspace $A\subset X$ the family $\mathcal N|A=\{N\cap A:N\in\mathcal N\}$ is a Pytkeev network for the space $A$. Consequently, subspaces of $\Pyt_0$-spaces are $\Pyt_0$-spaces.
\end{proposition}

Let us recall that the {\em topological sum} $\coprod_{\alpha\in A}X_\alpha$ of an indexed family $(X_\alpha)_{\alpha\in A}$ of topological spaces is the union $\bigcup_{\alpha\in A}A_\alpha\times\{\alpha\}$ endowed with the largest topology making every embedding $i_\alpha X_\alpha\to\coprod_{\alpha\in A}X_\alpha$, $i_\alpha:x\mapsto (x,\alpha)$, continuous.

\begin{proposition}\label{p3.2n} Let $(X_\alpha)_{\alpha\in A}$ be an indexed family of topological spaces with countable Pytkeev network. If $|A|\le\aleph_0$, then the topological sum $\coprod_{\alpha\in A}X_\alpha$ has a countable Pytkeev network.
\end{proposition}

\begin{proof} For every $\alpha\in A$ fix a countable Pytkeev network $\mathcal N_\alpha$ for the space $X_\alpha$. It is easy to see that the countable family
$$\mathcal N=\bigcup_{\alpha\in A}\{N\times\{\alpha\}:N\in \mathcal N_\alpha\}$$is a Pytkeev network for the topological sum $\coprod_{\alpha\in A}X_\alpha$.
\end{proof}

More generally, we show that the class of $\Pyt_0$-spaces is closed under taking inductive topologies.
We shall say that a topological space $X$ {\em carries the inductive topology with respect to a family}  $\mathcal C$ of its subspaces if a subset $U\subset X$ is open in $X$ if and only if for every $C\in\mathcal C$ the intersection $C\cap U$ is open in the subspace topology of $C$.
Observe that a topological space is a $k$-space if and only if it carries the inductive topology with respect to a cover by compact subsets. A topological space is a {\em $k_\w$-space} if it carries the inductive topology with respect to some countable cover by compact subsets.

\begin{theorem}\label{t3.2} A $T_1$-space $X$ has a countable Pytkeev network if $X$ carries the inductive topology with respect to a countable cover $\C$ of $X$ by subspaces possessing countable Pytkeev networks.
\end{theorem}

\begin{proof} For every space $C\in\C$ fix a countable Pytkeev network $\mathcal N_C$.  We claim that the countable family $\mathcal N=\bigcup_{C\in\C}\mathcal N_C$ is a Pytkeev network for the space $X$. First observe that $\mathcal N$ is a network for $X$.

Given an open set $U\subset X$, and a subset $A\subset X$ accumulating at a point $x\in U$, we should find a set in the family $\mathcal N_U=\{N\in\mathcal N:N\subset U\}$ that contains infinitely many points of $A$.   Replacing the set $A$ by the set  $(A\setminus\{x\})\cup (X\setminus U)$, we can additionally assume that $x\notin A$ and $X\setminus U\subset A$. Since the set $A$ is not closed in $X$, for some $C\in\C$ the intersection $A\cap C$ is not closed in $C$ and hence has some accumulation point $c\in C\setminus A\subset C\cap U$.
The family $\mathcal N_C$, being a Pytkeev network for the space $C$, contains a set $N\in\mathcal N_C\subset\mathcal N$ such that $N\subset U\cap C$ and $N\cap (A\cap C)\subset N\cap A$ is infinite.
\end{proof}

A topological space $X$ is {\em submetrizable} if it admits a continuous metric. Observe that each submetrizable $k_\w$-space carries inductive topology with respect to a countable cover by compact metrizable subspaces. Combining this fact with Theorem~\ref{t3.2} we get:

\begin{corollary}\label{koPi} Each submetrizable $k_\w$-space is a $\Pyt_0$-space.
\end{corollary}

Next, we show that the class of $\Pyt_0$-spaces is preserved by countable Tychonoff products.

\begin{theorem}\label{t3.4} Let $(X_\alpha)_{\alpha\in A}$ be an indexed family of topological spaces with countable Pytkeev network. If $|A|\le\aleph_0$, then the Tychonoff product $\prod_{\alpha\in A}X_\alpha$ has a countable Pytkeev network.
\end{theorem}

\begin{proof} By Proposition~\ref{p3.2n}, the topological sum $X=\coprod_{\alpha\in A}X_\alpha$ of the spaces $X_\alpha$, $\alpha\in A$, has a countable Pytkeev network. Observe that the set $A$ endowed with the discrete topology is an $\aleph_0$-space. By Theorem~\ref{main}, the function space $C_k(A,X)$ has countable Pytkeev network. Since the function space $C_k(A,X)$ is homeomorphic to the Tychonoff power $X^A$, we conclude that $X^A$ as well as its subspace $\prod_{\alpha\in A}X_\alpha$ have countable Pytkeev networks.
\end{proof}

Let us recall that the {\em box product} $\square_{\alpha\in A}X_\alpha$ of topological spaces $X_\alpha$, $\alpha\in A$, is the Cartesian product $\prod_{\alpha\in A}X_\alpha$ endowed with the topology generated by the base consisting of the products $\prod_{\alpha\in A}U_\alpha$ of open sets $U_\alpha\subset X_\alpha$.

A {\em pointed topological space} is a topological space $X$ with a distinguished point $*_X$.

By the {\em small box-product} of an indexed family $(X_\alpha)_{\alpha\in A}$ of pointed topological spaces we understand the subspace
$$\cbox_{\alpha\in A}X_\alpha=\big\{(x_\alpha)_{\alpha\in A}\in\square_{\alpha\in A}X_\alpha:\mbox{the set $\{\alpha\in A:x_\alpha\ne *_{X_\alpha}\}$ is finite}\big\}$$
of the box-product $\square_{\alpha \in A}X_\alpha$.

\begin{theorem}\label{t3.5} If pointed $T_1$-spaces $X_n$, $n\in\w$, have countable Pytkeev networks, then their small box-product $\cbox_{n\in\w}X_n$ has a countable Pytkeev network too.
\end{theorem}

\begin{proof} In the small box-product $\cbox_{n\in\w}X_n$ consider the subspace
$$X=\{(x_n)_{n\in\w}\in\cbox_{n\in\w}X_n:|\{n\in\w:x_n\ne *_{X_n}\}|\le1\}$$
which can be thought as the bouquet of the pointed spaces $X_n$, $n\in\w$.
Here we identify every space $X_k$ with the  subspace $\{(x_n)_{n\in\w}\in \cbox_{n\in\w}X_n:\forall n\in\w\setminus\{k\}\;\;x_n=*_{X_n}\}$ of $\cbox_{n\in\w}X_n$.
The point $*_X=(*_{X_n})_{n\in\w}$ is the distinguished point of the space $X$.

It is easy to check that the subspace topology on $X$ is inductive with respect to the cover $\{X_n\}_{n\in\w}$. By Theorem~\ref{t3.2}, the $T_1$-space $X$ has a countable Pytkeev network. Consider the compact subspace $S=\{0\}\cup\{2^{-n}:n\in\w\}$ of the real line and observe that the small box-product $\cbox_{n\in\w}X_n$ is homeomorphic to the subspace
$$\{f\in C_k(S,X):f(0)=*_X\mbox{ \ and \ }\forall n\in\w\;\;f(2^{-n})\in X_n\}$$of the function space $C_k(S,X)$, which has a countable Pytkeev network according to Theorem~\ref{main}. Then the small box-product $\cbox_{n\in\w}X_n$ also has a countable Pytkeev network, being homeomorphic to a subspace of $C_k(S,X)$.
\end{proof}

For a topological space $X$ by $\exp(X)$ we shall denote the space of non-empty compact subsets endowed with the Vietoris topology generated by the base consisting of the sets $$\langle U_1,\dots,U_n\rangle=\{K\in\exp(X):K\subset U_1\cup\dots\cup U_n\mbox{ \ and \ }K\cap U_i\ne\emptyset\mbox{ for all $i\le n$}\}$$where $U_1,\dots,U_n$ run over non-empty open sets in $X$.

\begin{theorem}\label{t3.6} If a Hausdorff topological space $X$ has a countable Pytkeev network, then its hyperspace $\exp(X)$ has a countable Pytkeev network too.
\end{theorem}

\begin{proof} For the space $X$ fix a countable Pytkeev network $\mathcal P$ which is closed under finite unions. We claim that the countable family $\tilde{\mathcal P}$ consisting of the sets
$$\langle P_1,\dots,P_n\rangle=\{K\in\exp(X):K\subset P_1\cup\dots\cup P_n\mbox{ \ and \ }K\cap P_i\ne\emptyset\mbox{ for all $i\le n$}\}$$where $P_1,\dots,P_n\in\mathcal P$ is a strict Pytkeev network for $\exp(X)$.

Fix a compact set $K\in\exp(X)$, a neighborhood $O_K\subset\exp(X)$ and a set $\A\subset\exp(X)$ containing $K$ in its closure. We lose no generality assuming that the neighborhood $O_K$ is of the basic form $O_K=\langle U_1,\dots,U_m\rangle$ for some open sets $U_1,\dots,U_m\subset X$. Replacing $\A$ by $\A\cap O_K$ we can also assume that $\A\subset O_K$. Since $K$ is a compact Hausdorff space, we can find closed sets $K_1,\dots,K_m\subset K$ such that $K=\bigcup_{i=1}^mK_i$ and $K_i\subset U_i$ for all $i\in\{1,\dots,m\}$.

For every $i\in\{1,\dots,m\}$ put $\mathcal P(K_i,U_i)=\{P\in\mathcal P:K_i\subset P\subset U_i\}$. Let $\mathcal P(K_i,U_i)=\{P'_{i,n}\}_{n\in\w}$ be an enumeration of the countable family $\mathcal P(K_i,U_i)$ and for every $n\in\w$ let $P_{i,n}=\bigcup_{j\le n}P'_{i,j}$. Since the Pytkeev network $\mathcal P$ is closed under finite unions, the sets $P_{i,j}$, $j\in\w$, belong to $\mathcal P(K_i,U_i)$ and form an increasing sequence $(P_{i,j})_{j\in\w}$ with $\bigcup_{j\in\w}P_{i,j}=U_i$.

For every $j\in\w$ consider the set
$$\tilde P_{j}=\langle P_{1,j},\dots,P_{m,j}\rangle\in\tilde{\mathcal P}$$and observe that $\tilde P_j\subset\tilde P_{j+1}$.

\begin{claim}\label{hyp1} $O_K=\bigcup_{j\in\w}\tilde P_j$.
\end{claim}

\begin{proof} Given any compact set $C\in O_K$, we need to find a number $j\in\w$ with $C\in\tilde P_j$. It follows that $C\subset \bigcup_{i=1}^m U_i$ and $C\cap U_j\ne\emptyset$ for every $j\in\{1,\dots,m\}$. Since $C$ is a compact Hausdorff space, we can find non-empty closed subsets $C_1,\dots,C_m\subset C$ such that $C_i\subset U_i$ for all $i\le m$. By Proposition~\ref{p1.2}, the countable Pytkeev network $\mathcal P$ is a $k$-network for $X$. Consequently, there is $j\in\w$ such that $C_k\subset P_{k,j}\subset U_k$ for every $k\le m$. Then $C=\bigcup_{i=1}^nC_i\in\langle P_{1,j},\dots,P_{m,j}\rangle=\tilde P_j$.
\end{proof}

\begin{claim}\label{hyp2} If $K$ is an accumulation point of the set $\A$, then for some $j\in\w$ the intersection $\A\cap\tilde P_j$ is infinite.
\end{claim}

\begin{proof} Assume that for every $j\in\w$ the intersection $\A_j=\A\cap\tilde P_j$ is finite. Since $\A\subset O_K=\bigcup_{j\in\w}\tilde P_j$, the set $\A=\bigcup_{j\in\w}\A_j$ is the countable union of finite sets $\A_j$. For every compact set $\alpha\in\A\setminus\A_0$ find a unique number $j_\alpha\in\w$ such that $\alpha\in\A_{j_\alpha+1}\setminus\A_{j_\alpha}$. Then $\alpha\notin \tilde P_{j_\alpha}$ and hence either $\alpha\not\subset \bigcup_{i=1}^mP_{i,j_\alpha}$ or $\alpha\cap P_{i,j_\alpha}=\emptyset$ for some $i\in\{1,\dots,m\}$.

If $\alpha\not\subset \bigcup_{i=1}^mP_{i,j_\alpha}$, then put $i_\alpha=0$ and choose any point $x_\alpha\in \alpha\setminus\bigcup_{i=1}^mP_{i,j_\alpha}$.

If $\alpha\subset\bigcup_{i=1}^mP_{i,j_\alpha}$, then fix any number $i_\alpha\in \{1,\dots,m\}$ such that $\alpha\cap P_{i_\alpha,j_\alpha}=\emptyset$ and choose a point $x_\alpha\in \alpha\cap U_{i_\alpha}$.

Let $\A(0)=\{\alpha\in\A\setminus\A_0:i_\alpha=0\}$ and
$B_0=\{x_\alpha:\alpha\in \A(0)\}\cup \big(X\setminus \bigcup_{i=1}^mU_i\big)$.

For every $i\in\{1,\dots,m\}$ let $\A(i)=\{\alpha\in\A\setminus\A_0:i_\alpha=i\}$ and
$B_i=\{x_\alpha:\alpha\in\A(i)\}\cup (X\setminus U_i)$.

It follows that $\A\setminus\A_0=\bigcup_{i=0}^m\A(i)$.

\begin{subclaim} The set $B_0$ is closed in $X$.
\end{subclaim}

\begin{proof} Assuming that $B_0$ is not closed, we would find an accumulation point $y\in \bar B_0\setminus B_0$. Since $X\setminus\bigcup_{i=1}^mU_i\subset B_0$, we conclude that $y\in U_i$ for some $i\in\{1,\dots,m\}$. Since $\mathcal P$ is a strict Pytkeev network for $X$, there is a set $P\in\mathcal P$ such that $y\in P\subset U_i$ and $P\cap B_0$ is infinite. The Pytkeev network $\mathcal P$, being a $k$-network closed under unions, contains a set $P_{i,j}\in\mathcal P(K_i,U_i)$ such that $P\subset P_{i,j}$. Then the set $P_{i,j}$ has infinite intersection with the set $B_0\cap U_i\subset\{x_\alpha:\alpha\in\A(0)\}$ which is not possible as $x_\alpha\notin\bigcup_{i=1}^mP_{i,j}$ for $\alpha\in\A(0)\setminus\A_j$. This contradiction shows that the set $B_0$ is closed in $X$.
\end{proof}

\begin{subclaim} For every $i\in\{1,\dots,m\}$ the set $B_i$ is closed in $X$.
\end{subclaim}

\begin{proof} Assuming that $B_i$ is not closed, we can find an accumulation point $y\in\bar B_i\setminus B_i$. Since $X\setminus U_i\subset B_i$, the point $y$ belongs to $U_i$. Since $\mathcal P$ is a Pytkeev network and $k$-network, there is a set $P\in\mathcal P(K_i,U_i)$ such that the intersection $P\cap B_i$ is infinite. It follows that $P\subset P_{i,j}$ for some $j\in\w$. Since $P\cap B_i\subset P_{i,j}\cap \{x_\alpha:\alpha\in\A(i)\}\subset \{x_\alpha:\alpha\in\A(i)\cap\A_j\}$ the intersection $P\cap B_i$ is finite. This contradiction shows that the set $B_i$ is closed.
\end{proof}

The choice of the points $x_\alpha$, $\alpha\in\A$, guarantees that $B_0\cap K=\emptyset$ and $K_i\cap B_i=\emptyset$ for all $i\in\{1,\dots,m\}$. For every $i\in\{1,\dots,m\}$ let $V_i=U_i\setminus (B_i\cup B_0)$ and observe that $\langle V_1,\dots,V_m\rangle$ is an open neighborhood of $K$ in the hyperspace $\exp(X)$, which is disjoint with the set $\A\setminus\A_0=\bigcup_{i=0}^m\A(i)$. This implies that $K$ is not an accumulation point of the set $\A$, and this is a desired contradiction.
\end{proof}

Claims~\ref{hyp1}, \ref{hyp2} complete the proof of Theorem~\ref{t3.6}.
\end{proof}

Proposition~\ref{p3.1}, \ref{p3.2n} and Theorems~\ref{t3.4}, \ref{t3.5}, \ref{t3.6}, \ref{main} imply the following stability result for the class of $\Pyt_0$-spaces.

\begin{corollary} The class of $\Pyt_0$-spaces is closed under taking subspaces, countable topological sums, countable Tychonoff products, countable small box-products,  the hyperspaces, and the function spaces endowed with the compact-open topology.
\end{corollary}

Finally we detect free abelian topological groups and free locally convex linear topological spaces which are $\Pyt_0$-spaces.

By a {\em free abelian topological group} over a topological space $X$ we understand a pair $(A(X),i_X)$ consisting of an abelian topological group $A(X)$ and a continuous map $i_X:X\to A(X)$ such that for every map $f:X\to G$ to an abelian topological group $G$ there exists a unique continuous group homomorphism $h:A(X)\to G$ such that $f=h\circ i_X$.

A {\em free locally convex space} over a topological space $X$ is a pair $(L(X),j_X)$ consisting of a locally convex linear topological space $L(X)$ and a continuous map $j_X:X\to L(X)$ such that for every map $f:X\to Y$ to a locally convex linear topological space $Y$ there exists a unique linear continuous operator $\lambda:L(X)\to Y$ such that $f=\lambda\circ j_X$.

Since $L(X)$ is an abelian topological group, there is a unique continuous group homomorphism $h:A(X)\to L(X)$ such that $h\circ i_X=j_X$. By \cite{Tk83}, the homomorphism $h:A(X)\to L(X)$ is a topological embedding.

For every topological space $X$ the function space $C_k(C_k(X))$ is a locally convex linear topological space. Each point $x\in X$ can be identified with the {\em Dirac measure} $\delta_x:C_k(X)\to \IR$, which assigns to each function $\varphi\in C_k(X)$ its value $\varphi(x)$. Thus we define the canonical map $\delta_X:X\to C_k(C_k(X))$
, $\delta_X:x\mapsto\delta_x$. The map $\delta_X$ is continuous if $X$ is a $k$-space and is a topological embedding if $X$ is a Tychonoff $k$-space (this follows from Lemma 3.4.18 and Ascoli Theorem 3.4.20 in [En]).

It follows that for a Tychonoff $k$-space there is a unique linear operator $\lambda:L(X)\to C_k(C_k(X))$ such that $\lambda\circ j_X=\delta_X$. By \cite{Us83} or \cite{Flood}, the linear operator $\lambda$ is a topological embedding. Combining this fact with Corollary~\ref{tipa2} (applied two times), we obtain the following result first observed by A.Leiderman.

\begin{theorem}[Leiderman]\label{t:leider} For any sequential $\aleph_0$-space $X$ the abelian topological group $A(X)$ and the free locally convex linear topological space $L(X)$ both are $\Pyt_0$-spaces.
\end{theorem}

\begin{problem} Let $X$ be a (sequential) $\Pyt_0$-space. Is the free topological group over $X$ a $\Pyt_0$-space? Is the free linear topological space over $X$ a $\Pyt_0$-space? {\rm (The answer to both questions is affirmative if $X$ is a submetrizable $k_\w$-space).}
\end{problem}

\section{Detecting $\Pyt_0$-spaces among topological spaces endowed with compatible algebraic structures}\label{s4}

In this section we shall detect $\Pyt_0$-spaces among topological spaces endowed with compatible algebraic structures.
First we recall some algebraic notions. The most basic one is that of  {\em magma}, i.e., a set $X$ endowed with a binary operation $\cdot:X\times X\to X$ which assigns to any pair of points $(x,y)\in X\times X$ their product $xy$.
Any point $a$ of a magma $X$ determines two maps $L_a:X\to X$, $L_a:x\mapsto ax$, and $R_a:X\to X$, $R_a:x\mapsto xa$, called the {\em left shift} and the {\em right shift} by $a$, respectively. An element $e$ of a magma is called a {\em unit} if $xe=x=ex$ for all $x\in X$. This is equivalent to saying that the left shift $L_e$ and the right shift $R_e$ both coincide with the identity map of $X$. It is standard to show that any two units in a magma coincide, so a magma can have at most one unit.

Now we recall the definitions of some classes of magmas.

\begin{definition}A magma $X$ is called
\begin{itemize}
\item {\em associative} if $(xy)z=x(yz)$ for any points $x,y,z\in X$;
\item a {\em semigroup} if $X$ is an associative magma;
\item a {\em monoid} if $X$ is a semigroup with unit;
\item a {\em group} if $X$ is a semigroup with unit $e$ such that for each $x\in X$ there is an element $x^{-1}\in X$ such that $xx^{-1}=e=x^{-1}x$;
\item a {\em loop} if $X$ has a unit and for every $a\in X$ the left shift $L_a:X\to X$ and the right shift $R_a:X\to X$ both are bijective;
\item a {\em lop} if $X$ has a unit and for every $a\in X$ the left shift $L_a:X\to X$ is bijective.
\end{itemize}
\end{definition}
These notions relate as follows:
$$
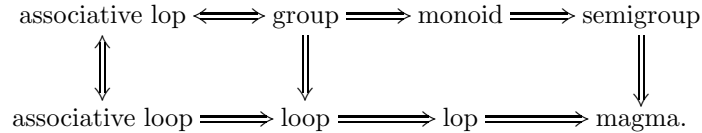
\xymatrix{
\mbox{associative lop}\ar@{<=>}[r]\ar@{<=>}[d]&\mbox{group}\ar@{=>}[r]\ar@{=>}[d]&\mbox{monoid}\ar@{=>}[r]&\mbox{semigroup}\ar@{=>}[d]\\
\mbox{associative loop}\ar@{=>}[r]&
\mbox{loop}\ar@{=>}[r]&\mbox{lop}\ar@{=>}[r]&\mbox{magma.}\\
}
$$

\begin{remark} The notions of semigroup, monoid, group and loop are well-known and standard in Algebra. Magmas were introduced by Bourbaki in \cite{Brb}. For magmas there is an alternative term ``groupoid'' which however has many other different meanings (especially in Category Theory). So, we prefer to use magmas stressing its basic nature for all other algebraic notions. The notion of lop seems to be new. It is an intermediate notion between those of loop and left-loop (i.e., a magma $X$ with right unit and bijective left shifts).
\end{remark}

For our purposes the basic notion is that of lop. As we already know, the class of lops include all loops and all groups. For a lop $X$ its binary operation will be denoted by $p:X\times X\to X$, $p:(x,y)\mapsto xy$, and called the {\em multiplication operation} on $X$. By definition of a lop, each left shift $L_x:X\to X$, $x\in X$, is bijective. So, we can consider its inverse $L_x^{-1}$ and define the binary operation $q:X\times X\to X$, $q:(x,y)\mapsto L_x^{-1}(y)$, which will be called the {\em division operation} on $X$. It will be convenient to denote the element $q(x,y)=L_x^{-1}(y)$ by $x^{-1}y$.  The equalities $L_x\circ L_x^{-1}=\mathrm{id}=L_x^{-1}\circ L_x$ imply that $x(x^{-1}y)=y=x^{-1}(xy)$ for all $x,y\in X$. So, lops can be considered as universal algebras with one $0$-ary operation $X^0\to \{e\}\subset X$ selecting the unit $e$ of $X$ and two binary operations $p:X\times X\to X$, $p:(x,y)\mapsto xy$ and $q:X\times X\to X$, $q(x,y)\mapsto x^{-1}y$, satisfying the identities: $xe=x=ex$, $x^{-1}(xy)=y=x(x^{-1}y)$ for all $x,y\in X$. For two subsets $A,B$ of a lop $X$ let $AB=\{ab:a\in A,\;b\in B\}$ and $A^{-1}B=\{a^{-1}b:a\in A,\;b\in B\}$ be the results of their pointwise multiplication and division in the lop $X$.

Now we shall define some compatible topological structures on lops. By a {\em topologized lop} we  understand a lop endowed with a topology.

\begin{definition} A topologized lop is called
\begin{itemize}
\item a {\em semitopological lop} if the multiplication $p:X\times X\to X$ is separately continuous and all left shifts $L_a:X\to X$, $a\in X$, are homeomorphisms;
\item a {\em quasitopological lop} if the multiplication $p:X\times X\to X$ and the division $q:X\times X\to X$ both are separately continuous;
\item a {\em paratopological lop} if $X$ is a semitopological lop with continuous multiplication $p:X\times X\to X$;
\item a {\em topological lop} if the multiplication $p:X\times X\to X$ and division $q:X\times X\to X$ both are continuous;
\item a {\em topological group} (resp. {\em semitopological, quasitopoloical, paratopological} {\em group}) is $X$ is an associative  topological lop (resp. semitopological, quasitopological, paratopological lop).
    \end{itemize}
\end{definition}
All these notions relate as follows:
$$
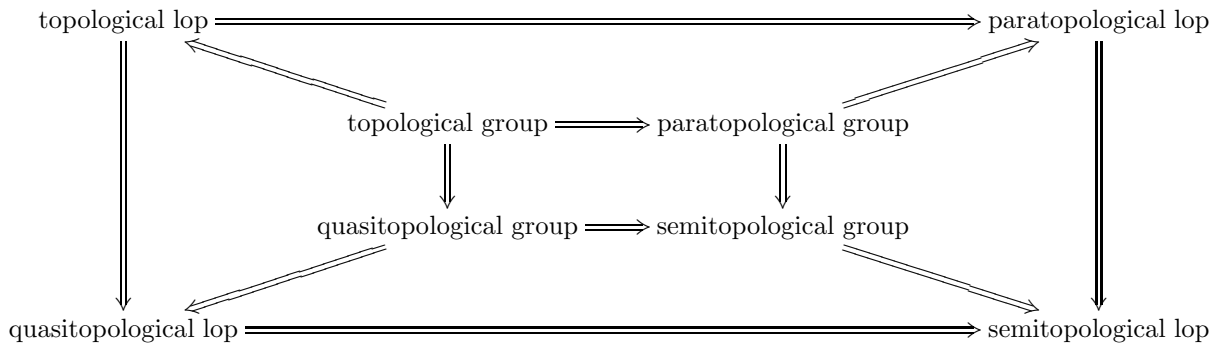
\xymatrix{
\mbox{topological lop}\ar@{=>}[rrr]\ar@{=>}[ddd]&
&&\mbox{paratopological lop}\ar@{=>}[ddd]\\
&\mbox{topological group}\ar@{=>}[r]\ar@{=>}[ul]\ar@{=>}[d]&
\mbox{paratopological group}\ar@{=>}[d]\ar@{=>}[ru]\\
&\mbox{quasitopological group}\ar@{=>}[r]\ar@{=>}[dl]&
\mbox{semitopological group}\ar@{=>}[dr]\\
\mbox{quasitopological lop}\ar@{=>}[rrr]&
&&\mbox{semitopological lop}
}
$$

\begin{remark}
The notion of topological (semitopological, quasitopological, paratopological) lop generalizes the notion of topological (semitopological, quasitopological, paratopological) group, well-known in Topological Algebra (see \cite{ArT}).
The class of topological lops is contained in the class of topological left-loops introduced recently by Hofmann and Martin \cite{HM}. A {\em topological left-loop} is a topological space $X$ endowed with a distinguished point $e$ and two continuous binary operations $p:X\times X\to X$, $p:(x,y)\mapsto xy$, and $q:X\times X\to X$, $q:(x,y)\mapsto x^{-1}y$ such that $xe=x$ and $x(x^{-1}y)=y=x^{-1}(xy)$ for all $x,y\in X$. A topological left-loop $(X,e)$ is a topological lop if and only if $ex=x$ for all $x\in X$.
\end{remark}

\begin{remark} It is well-known that if the topological space of a topological group $G$ satisfies the separation axiom $T_0$, then it is Tychonoff (i.e., satisfies the separation axiom $T_{3\frac12}$). For topological lops a bit weaker result is true: each topological lop satisfying the separation axiom $T_0$ is regular (i.e., is a $T_3$-space). This follows from Corollary 2.2 \cite{Gul} (saying that each rectifiable $T_0$-space is regular) and Proposition~\ref{p:reclop} (saying that each topological lop is a rectifiable space). So, from now on we assume that all topological lops satisfy the separation axiom $T_0$ and hence are regular spaces. It is not known if topological lops are automatically Tychonoff (see Problem 4.18 \cite{Ar02}).
\end{remark}

The following theorem detects $\Pyt_0$-spaces among paratopological lops.

\begin{theorem}\label{t:paraPi} A paratopological lop $X$ is a $\Pyt_0$-space if and only if $X$ is a cosmic space with the strong Pytkeev property.
\end{theorem}

\begin{proof} The ``only if'' part of this theorem is trivial. To prove the ``if'' part, assume that $\mathcal P$ is a countable Pytkeev network at the unit $e$ of $X$ and $\mathcal N$ is a countable network for $X$.
Replacing the family $\mathcal P$ by $\mathcal P\vee\mathcal P=\{A\cup B:A,B\in\mathcal P\}$, we can assume that $\mathcal P$ is a strict Pytkeev network at $e$. For two sets $A,B\subset X$, by $AB=\{ab:a\in A,\;b\in B\}$ we denote their pointwise product in the lop $X$.

 We claim that the countable family
$$\mathcal N\mathcal P=\{NP:N\in\mathcal N,\;P\in\mathcal P\}$$is a strict
Pytkeev network for the space $X$.
Given a set $A\subset X$, a point $x\in\bar A$ and a neighborhood $O_x\subset X$ of $x$ we need to find sets $N\in\mathcal N$ and $P\in\mathcal P$ such that $x\in NP\subset U$ and moreover $A\cap NP$ is infinite if $x$ is an accumulation point of $A$.

Using the continuity of the multiplication $p:X\times X\to X$ at $(x,e)$, we can find neighborhoods $U_x\subset X$ of $x$ and $U_e\subset X$ of $e$ such that $U_x U_e:=p(U_x\times U_e)\subset O_x$.
Since $\mathcal N$ is a network in $X$, there is a set $N\in\mathcal N$ such that $x\in N\subset U_x$.

Since the left shift $L_x:X\to X$ is a homeomorphism, the point $e=L_x^{-1}(x)=x^{-1}x$ belongs to the set $x^{-1}A:=L_x^{-1}(A)$. Moreover $e$ is an accumulation point of $x^{-1}A$ if and only if $x$ is  an accumulation point of $A$. Since $\mathcal P$ is a strict Pytkeev network at $e$, there is a set $P\in\mathcal P$ such that $e\in P\subset U_x$ and moreover $P\cap x^{-1}A$ is infinite if $x^{-1}A$ accumulates at $e$. Then the set $NP\in\mathcal N\mathcal P$ has the desired property: $x\in NP\subset U_xU_e\subset O_x$. Moreover, if $x$ is an accumulating point of $A$, then $e$ is accumulating point of $x^{-1}A$, $P\cap x^{-1}A$ is infinite and so are
the sets $x(P\cap x^{-1}A)=xP\cap A\subset NA\cap A$.
\end{proof}

For topological lops the cosmicity in Theorem~\ref{t:paraPi} can be weakened to separability as shown in the following theorem inspired by a result of Gabriyelyan and K\c akol \cite{GK2} (see Corollary~\ref{c:gab}).

\begin{theorem}\label{t:Saaklop} A topological lop $X$ has a countable Pytkeev network if and only if $X$ is a separable space with the strong Pytkeev property.
\end{theorem}

\begin{proof} The ``only if'' part of this theorem is trivial. To prove the ``if'' part, assume that $X$ is a separable space with the strong Pytkeev property. Fix a countable dense subset $D\subset X$ and a  countable Pytkeev network $\mathcal P$ at the unit $e$ of $X$. We lose no generality assuming that $\mathcal P$ is closed under finite unions and each set $P\in\mathcal P$ contains the intersection $\ddot e$ of all neighborhoods of $e$ in $X$.

 We claim that the countable family
$$\mathcal N=\{(xP)Q:x\in D,\;P,Q\in\mathcal P\}$$is a strict
Pytkeev network for the space $X$.
Given a set $A\subset X$, a point $a\in\bar A$ and a neighborhood $O_a\subset X$ of $a$ we need to find a point $x\in D$ and two sets $P,Q\in\mathcal P$ such that $a\in (xP)Q\subset O_x$ and moreover $A\cap (xP)Q$ is infinite if $a$ is an accumulation point of $A$.
Since $(ae)e=a$, by the continuity of the multiplication, we can find open sets $U_a\ni a$ and $U_e\ni e$ in $X$ such that $(U_aU_e)U_e\subset O_a$.

Consider the subfamily $\mathcal P'=\{P\in\mathcal P:P\subset U_e\}$. We claim that $\bigcup\mathcal P'$ is a neighborhood of $e$ in $X$. Assuming the opposite, we conclude that the set $B=X\setminus \bigcup\mathcal P'$ contains the point $e$ it its closure. Since $\mathcal P'$ is a Pytkeev network at $e$, the set $B$ cannot accumulate at $e$. So, there is a neighborhood $W$ of $e$ such that $W\cap B$ is finite. Since the intersection $\ddot e$ of all neighborhoods of $e$ misses the set $B$, we can find a neighborhood $W'\subset W$ of $e$ such that $W'\cap (W\cap B)=\emptyset$. Then the neighborhood $W_e=W\cap W'$ of $e$ is disjoint with the set $B$ and hence is contained in $\bigcup\mathcal P'\subset U_e$. Since $a^{-1}a=e$, we can use the (separate) continuity of division in $X$ and find a neighborhood $W_a\subset U_a$ of $a$ such that $W_a^{-1}a\subset W_e$. Finally, choose any point $x\in D\cap W_a$ and observe that $x^{-1}a\in W_e\subset\bigcup\mathcal P'$ and hence $x^{-1}a\in P$ for some set $P\in\mathcal P'$. Then $a\in xP$. Now consider the left shift $L_a:X\to X$, $x\mapsto ax$, and observe that the point $L_a^{-1}(a)=a^{-1}a=e$ belongs to the closure of the set $L_a^{-1}(A)=a^{-1}A$. Moreover, $e$ is an accumulating point of the set $a^{-1}A$ if and only if $a$ is an accumulation point of $A$. Since $\mathcal P$ is a Pytkeev network at $e$, there is a set $Q\in\mathcal P$ such that $e\in Q\subset U_e$ and  $Q\cap a^{-1}A$ is infinite if $e$ is an accumulation point of $a^{-1}A$. Now consider the set $(xP)Q\in\mathcal N$ and observe that $a\in (xP)e\subset (xP)Q\subset (U_aU_e)U_e\subset O_e$. If $a$ is an accumulation point of $A$, then $e$ is an accumulation point if $a^{-1}A$ and hence the intersections $Q\cap a^{-1}A$ and $aQ\cap A$ are infinite. Since $aQ\subset (xP)Q$, the intersection $(xP)Q\cap A$ is infinite as well.
\end{proof}

For (para)topological groups Theorems~\ref{t:paraPi} and \ref{t:Saaklop} imply the following two corollaries.

\begin{corollary}\label{c:Pigroup} A paratopological group is a $\Pyt_0$-space if and only if it is a cosmic space with the strong Pytkeev property.
\end{corollary}

\begin{corollary}[Gabriyelyan-K\c akol]\label{c:gab} A topological group is a $\Pyt_0$-space if and only it is separable and has the strong Pytkeev property.
\end{corollary}

Now let us present an example showing that Corollary~\ref{c:gab} does not hold for paratopological groups. We recall that the {\em Sorgenfrey line} is the real line $\IR$ endowed with the topology generated by the base consisting of half-intervals $[a,b)$, $a<b$. It is a classical example of a paratopological group which is not a topological group.

\begin{example} The Sorgenfrey line is a paratopological group having the following properties:
 \begin{enumerate}
 \item $X$ is first countable (and hence has the strong Pytkeev property);
 \item $X$ is separable but not cosmic (and hence not a $\Pyt_0$-space).
 \end{enumerate}
 \end{example}

On the other hand, Corollary~\ref{c:Pigroup} does not hold for quasitopological groups.

\begin{example}\label{cosmicex} There is a first countable cosmic quasitopological group $G$ which fails to be an $\aleph_0$-space.
\end{example}

\begin{proof} Let $\IQ$ be the additive group of rational numbers. Endow the group $G=\IR\times\IQ$ with the shift-invariant topology $\tau$ whose neighborhood base at zero $(0,0)$ consists of the sets
$$\maltese_\e=\{(0,0)\}\cup\{(x,y)\in\IR\times\IQ:|xy|<\e(x^2+y^2)<\e^2\}$$
where $\e>0$. It is easy to see that this topology is regular and the family $\{(a,b)\times\{q\}\colon a,b,q\in\IQ,\;a<b\}$ is a countable network for $G$. Since the topology $\tau$ is first countable and invariant under the inversion, the group $G=\IR\times\IQ$ endowed with the topology $\tau$ is a first countable cosmic quasitopological group.

It remains to show that $G$ has no countable $k$-network. Assume conversely that $\mathcal P$ is a countable $k$-network for $G$, closed under finite unions. Choose a positive $\e<\frac12$ such that  $\maltese_\e\subset\{(x,y)\in\IR\times\IQ: |y|\ge 2|x| \mbox{ or $|x|\ge 2|y|$}\}$. Next, choose any compact subset $K\subset\IQ$ such that $\{0\}\times K\subset\maltese_\e$ and for every $m\in\w$ the set $K$ meets the interval $(2^{-m-1},2^{-m})$.

Since $\mathcal P$ is a (closed under finite unions) $k$-network for $\IR\times\IQ$,
for every $x\in\IR$ there is a set $P_x\in\mathcal P$ such that $\{x\}\times K\subset P_x\subset (x,0)+\maltese_\e$. Since the family $\mathcal P$ is countable, by the Pigeonhole Principle,
for some $P\in\mathcal P$ the subset $X=\{x\in\IR:P_x=P\}\subset\IR$ is uncountable and hence contains two points $x,x'\in X$ with $|x-x'|<\e$. Find a unique $m\in\IN$ such that $2^{-m-1}<|x-x'|\le 2^{-m}$ and find a point $y\in K\cap (2^{-m-1},2^{-m})$. Then $(x',y)\in \{x'\}\times K\subset P_{x'}=P_x\subset(x,0)+\maltese_\e$ and hence $(x'-x,y)\in\maltese_\e$, which is not possible as
$$\frac12=\frac{2^{-m-1}}{2^{-m}}<\frac{|y|}{|x'-x|}<\frac{2^{-m}}{2^{-m-1}}=2$$according to the choice of $\e>0$.
\end{proof}

\begin{remark} The argument used in Example~\ref{cosmicex} gives a bit more. Namely, for any uncountable subgroup $H\subset \IR$ the subgroup $H\times\IQ$ of the quasitopological group $(\IR\times\IQ,\tau)$ constructed in Example~\ref{cosmicex} is a first countable cosmic quasitopological group failing to be an $\aleph_0$-space.
\end{remark}

Finally, we detect $\Pyt_0$-spaces among (locally) narrow topological lops.
First we recall the necessary definitions.

A subset $B$ of a topological lop $X$ is called
\begin{itemize}
\item {\em narrow} if for any neighborhood $U\subset X$ of the unit $e$, any infinite subset $A\subset B$ contains a point $a\in A$ such that the intersection $A\cap aU$ is infinite;
\item {\em bounded} if for any neighborhood $U\subset X$ of the unit $e$ there is a finite subset $F\subset B$ such that $B\subset FU$.
\end{itemize}
It is easy to see that each narrow subset $B$ in a topological lop $X$ is bounded.

A topological lop $X$ is called
\begin{itemize}
\item {\em locally narrow} if $X$ contains a narrow neighborhood of the unit $e$ of $X$;
\item {\em locally bounded} if $X$ contains a bounded neighborhood of the unit $e$ of $X$;
\item ({\em locally}) {\em precompact} if $X$ is topologically isomorphic to a sublop of a (locally) compact topological lop.
\end{itemize}
Next we recall three topological notions. A topological space $X$ is called
\begin{itemize}
\item {\em pseudocompact} if $X$ is Tychonoff and each continuous real-valued function on $X$ is bounded;
\item {\em countably-compact} if each infinite subset of $X$ has an accumulation point in $X$;
\item {\em locally countably-compact} if each point of $X$ has a countably-compact neighborhood in $X$.
\end{itemize}

For any topological lop these properties relate as follows:
$$
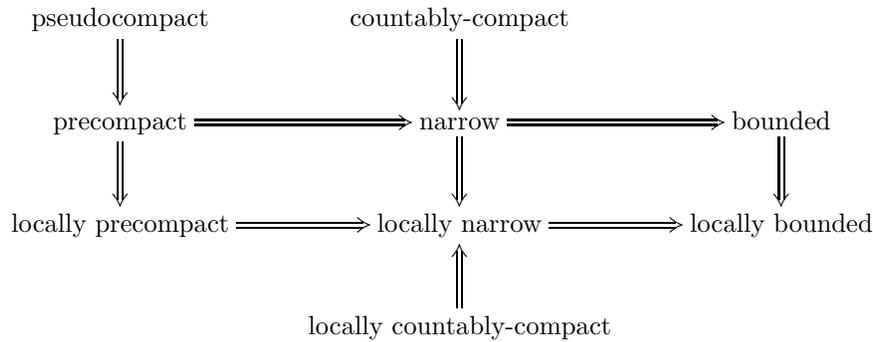
\xymatrix{
\mbox{pseudocompact}\ar@{=>}[d]&\mbox{countably-compact}\ar@{=>}[d]\\
\mbox{precompact}\ar@{=>}[d]\ar@{=>}[r]&\mbox{narrow}\ar@{=>}[d]\ar@{=>}[r]&\mbox{bounded}\ar@{=>}[d]\\
\mbox{locally precompact}\ar@{=>}[r]&\mbox{locally narrow}\ar@{=>}[r]&\mbox{locally bounded}\\
&\mbox{locally countably-compact}\ar@{=>}[u].
}
$$
\smallskip

Non-trivial implications from this diagram are proved in the following two propositions.

\begin{proposition}\label{p:lcclop} Each sublop $H$ of a (locally) countably-compact topological lop $G$ is (locally) narrow.
\end{proposition}

\begin{proof} Let $K$ be a countably-compact neighborhood of the unit $e$. If the space $G$ is countably compact, then we shall assume that $K=G$.  Since $e^{-1}e=e\in G$, by the continuity of the division $q:G\times G\to G$, $q:(x,y)\mapsto x^{-1}y$, there is a neighborhood $W\subset G$ of $e$ such that $W^{-1}W\subset K$. If $K=G$, then we put $W=G$. We claim that the neighborhood $W$  witnesses that the topological lop $G$ is (locally) narrow. Given a neighborhood $U\subset G$ of $e$ and an infinite subset $A\subset W$, we need to find a point $a\in A$ such that $A\cap aU$ is infinite. By the countable compactness of $K$, the set $A$ has an accumulation point $b\in K$. Since $b^{-1}b=e$, the continuity of the division operation $q:(x,y)\mapsto x^{-1}y$ yields a neighborhood $O_b\subset X$ of $b$ such that $O_b^{-1}O_b\subset U$. Fix any point $a\in A\cap O_b$ and observe that $a^{-1}O_b\subset U$ and hence $O_b\subset aU$. Since $b$ is an accumulation point of the set $A$, the intersection $A\cap O_b\subset A\cap aU$ is infinite.
\end{proof}

\begin{proposition} If $X$ is a pseudocompact topological lop, then
\begin{enumerate}
\item the Stone-\v Cech compactification $\beta X$ of $X$ has the structure of a topological lop
containing $X$ as a dense sublop;
\item each sublop of $X$ is narrow.
\end{enumerate}
\end{proposition}

\begin{proof} The operation $f:X^3\to X$, $f:(x,y,z)\mapsto x(y^{-1}z)$, has the properties $f(x,y,y)=x=f(y,y,x)$, which means that $f$ is a continuous Malcev operation on $X$ and $X$ is a Malcev space.
Since the Tychonoff product of pseudocompact Malcev space is pseudocompact \cite[1.6]{RU}, the product $X\times X$ is pseudocompact. Let $p:X\times X\to X$ and $q:X\times X\to X$ be the multiplication and division operations on the lop. Since $X\times X$ is pseudocompact, by
Glicksberg Theorem \cite[3.12.20(c)]{Eng}, the continuous map $\beta(X\times X)\to\beta X\times \beta X$ extending the identity embedding $X\times X\to\beta X\times \beta X$ is a homeomorphism. Consequently, the multiplication and division operations $p,q:X\times X\to X$ extend to continuous binary operations $\bar p:\beta X\times \beta X\to \beta X$ and $\bar q:\beta X\times \beta X\to\beta X$ on the Stone-\v Cech extension $\beta X$ of $X$. It follows that the closed subset
$$D=\{(x,y)\in\beta X\times\beta X:\bar p(x,e)=x=\bar p(e,x)\mbox{ \ and \ }\bar p(x,\bar q(x,y))=y=\bar q(x,\bar p(x,y))\}$$ contains the dense subset $X\times X$ of $\beta X\times\beta X$, and hence $D=\beta X\times \beta X$. This means that the operations $\bar p$ and $\bar q$ turn $\beta X$ into a compact topological lop.

By Proposition~\ref{p:lcclop}, any sublop $H\subset X$ is narrow (being a sublop of the compact topological lop $\beta X$).
\end{proof}

It is easy to see that a sublop of a locally narrow topological lop is locally narrow. We do not know if a similar property holds for locally bounded topological lops. But for topological groups we have the following equivalence (see \cite[3.7.I]{ArT}).

\begin{proposition} For any topological group $G$ the following conditions are equivalent:
\begin{enumerate}
\item $G$ is (locally) precompact;
\item $G$ is (locally) narrow;
\item $G$ is (locally) bounded.
\end{enumerate}
\end{proposition}

The following theorem detects $\Pyt_0$-spaces among locally narrow topological lops.

\begin{theorem}\label{t:lntlop} A locally narrow topological lop $X$ is a $\Pyt_0$-space if and only if $X$ is metrizable and separable.
\end{theorem}

\begin{proof} The ``if'' part is trivial. To prove the ``only if'' part, assume that $X$ is a $\Pyt_0$-space. Since the topological lop $X$ is locally narrow, there exists a neighborhood $W$ of the unit $e$ in $X$ such any for any neighborhood $U\subset X$ of $e$, each infinite set $A\subset W$ contains a point $a\in A$ such that $A\cap aU$ is infinite.

Fix any countable Pytkeev network $\mathcal N$ for $X$.
Since the $\Pyt_0$-space $X$ is regular, we can replace each set $N\in\mathcal N$ by its closure and assume that the Pytkeev network consists of closed subsets of $X$.
Consider the subfamily $\mathcal N'=\{N\in \mathcal N: N$ is nowhere dense in $X\}$ and let $\mathcal N'=\{N_k'\}_{k\in\w}$ be an enumeration of $\mathcal N'$. Use the nowhere density of the sets $N_k'$ to construct a sequence $A=\{a_k\}_{k\in\w}\subset W$ such that $a_n\in W\setminus\bigcup_{k,m<n}a_kN_m'$ for every $n\in\w$. We claim that the set $B=\{a_k^{-1}a_n:k<n\}\subset X$ accumulates at $e$. Indeed, given any neighborhood $U\subset X$ of $e$, we can find a point $a_k\in A$ such that the set  $A\cap a_kU$ is infinite. Then for every point $a_m\in A\cap a_kU$ with $m>k$ we get $a_k^{-1}a_m\in U\cap B$, which means that $B$ accumulates at $e$.
Observe that for any $m\in\w$ the set $B\cap N'_m\subset \{a_k^{-1}a_n:k<n\le m\}$ is finite.
This implies that the family $\mathcal N'=\{N_m\}_{m\in\w}$ is not a Pytkeev network at $e$.
Taking into account that $\mathcal N$ is a Pytkeev network at $e$, we conclude that each neighborhood $U\subset X$ of $e$ contains a set $N\in\mathcal N\setminus \mathcal N'$. Observe that each (closed) set $N\in\mathcal N\setminus\mathcal N'$ has non-empty interior and hence $N^{-1}N$ is a neighborhood of $e$. Then $\mathcal B_e=\{N^{-1}N:N\in\mathcal N\setminus\mathcal N'\}$ is a countable neighborhood base at $e$. So, $X$ is first countable at $e$. By \cite{Gul} the topological lop $X$, being rectifiable and first countable, is metrizable. Being cosmic, the space $X$ is separable.
\end{proof}

\begin{corollary}\label{c:lntgr} A locally narrow topological lop $X$ is metrizable and separable if and only if $X$ contains a dense sublop which is a $\Pyt_0$-space.
\end{corollary}

\begin{proof} The ``only if'' part is trivial. To prove the ``if'' part, assume that a locally narrow topological lop $X$ contains a dense sublop $H$ which is a $\Pyt_0$-space. The topological lop $H$ is locally narrow (as a sublop of a locally narrow topological lop). By Theorem~\ref{t:lntlop}, the topological lop $H$ is metrizable and separable. Taking into account that the first countable space $H$ is dense in the regular space $X$, we can conclude that the space $X$ is first countable at $e$. By \cite{Gul}, the topological space $X$, being a first countable rectifiable space, is metrizable.
Moreover the space $X$ is separable (since $X$ contains the dense separable subspace $H$).
\end{proof}

For locally precompact topological groups Corollary~\ref{c:lntgr} implies the following metrizability criterion.

\begin{corollary}[Gabriyelyan] A locally precompact topological group $G$ is metrizable and separable if and only if $G$ contains a dense subgroup which is a $\Pyt_0$-space.
\end{corollary}

\section{Detecting $\Pyt_0$-spaces among rectifiable and continuously homogeneous spaces}\label{s5}

Topological lops are tightly connected with rectifiable spaces, well-known in General Topology (see \cite[\S4]{Ar02}).

A topological space $X$ is called {\em rectifiable} if there is a point $e\in X$ and a homeomorphism $H:X\times X\to X\times X$ such that $H(x,e)=(x,x)$ and $H(\{x\}\times X)=\{x\}\times X$ for every $x\in X$.
The notion of rectifiable space was introduced by Arhangelski (in a Moscow topological seminar) and was studied in \cite{Gul}, \cite{Us}, \cite{Ar02}. Network properties of rectifiable spaces were studied in \cite{Lin}, \cite{LLL}, \cite{LS}. By Corollary 2.2 \cite{Gul}, each rectifiable $T_0$-space is regular.

It follows that each rectifiable space is topologically homogeneous.
Let us recall that a topological space $X$ is {\em topologically homogeneous} if for any points $x,y\in X$ there exists a homeomorphism $h_{x,y}:X\to X$ such that $h_{x,y}(x)\to y$. If the homeomorphism $h_{x,y}$ can be chosen to depend continuously on $x$ and $y$ (in the sense that the map $:X^3\to X^3$, $H:(x,y,z)\mapsto (x,y,h_{x,y}(z))$ is a homeomorphism), then the space $X$ is called {\em continuously homogeneous}. More precisely, a topological space $X$ is {\em continuously homogeneous} if there exists a homeomorphism $H:X^3\to X^3$ such that $H(x,y,x)=(x,y,y)$ and $H(\{(x,y)\}\times X)=\{(x,y)\}\times X$ for any points $x,y\in X$. Continuously homogeneous spaces were introduced by Uspenski\u\i\ \cite{Us} (who called them strongly homogeneous spaces). By \cite[Proposition 15]{Us}, a topological space $X$ is continuously homogeneous if and only if it is rectifiable.

The following simple proposition proved in \cite{BR} shows that topologically the notions of rectifiable space, continuously homogeneous space, topological left-loop and topological lop all are equivalent.

\begin{proposition}\label{p:reclop} For a topological space $X$ the following conditions are equivalent:
\begin{itemize}
\item $X$ is rectifiable;
\item $X$ is homeomorphic to a topological left-loop;
\item $X$ is homeomorphic to a topological lop;
\item $X$ is continuously homogeneous.
\end{itemize}
\end{proposition}

Proposition~\ref{p:reclop} combined with Theorem~\ref{t:Saaklop} implies the following characterization (proved in the realm of topological groups by Gabriyelyan):

\begin{corollary}\label{c:recsak} A rectifiable space $X$ is a $\Pyt_0$-space if and only if $X$ is a separable $T_0$-space with the strong Pytkeev property.
\end{corollary}

Next, we detect $\Pyt_0$-spaces among sequential rectifiable spaces.


\begin{theorem} For a sequential rectifiable $T_0$-space $X$ the following conditions are equivalent:
\begin{enumerate}
\item $X$ is a  $\Pyt_0$-space;
\item $X$ is a  $\aleph_0$-space;
\item $X$ is  Lindel\"of and has countable $\cs^*$-character;
\item $X$ is  separable and has the strong Pytkeev property;
\item $X$ is either metrizable separable or a submetrizable  $k_\w$-space.
\end{enumerate}
\end{theorem}

\begin{proof} The equivalence $(1)\Leftrightarrow (2)$ follows from Corollary~\ref{c1.7}, the implication $(1)\Ra(4)$ is trivial and $(4)\Ra(1)$ follows from Corollary~\ref{c:recsak}. This shows the equivalence of the conditions (1), (2) and (4).

The implication $(5)\Ra(1)$ follows from Corollary~\ref{koPi}, $(1)\Ra(3)$ is trivial and the most difficult implication $(3)\Ra(5)$ is proved in \cite{BR}.
\end{proof}

 \section{Acknowledgement}
The author would like to thank Arkady Leiderman and Saak Gabriyelyan for fruitful e-mail discussions which led to the current general form of Theorem~\ref{main} (initially this theorem was proved for metrizable separable spaces $X,Y$ thus answering a question from \cite{GKL}, then after Leiderman's suggestion it was generalized to the case of an $\aleph_0$-space $X$ and a metrizable separable space $Y$, and finally after a discussion with Saak Gabriyelyan it was generalized to its current final form). Besides that, the author expresses his thanks to Arkady Leiderman for the permission to include his Theorem~\ref{t:leider} in the text, and to Saak Gabriyelyan for the suggestion to use the German-style letter $\Pyt$ for denoting $\Pyt_0$-spaces (which were initially called ``Pytkeev $\aleph_0$-spaces'').

Also the author thanks Masami Sakai for his valuable comments turning Problem 2.12 (in the printed version of this paper) into Remark 2.12 (mentioning the results of M.Sakai on the countable fan-open tightness in function spaces).

\end{document}